\documentclass[oneside,english]{amsart}
\usepackage{lmodern}
\usepackage[T1]{fontenc}
\usepackage[latin9]{inputenc}
\usepackage{amsthm}
\usepackage{amssymb}
\PassOptionsToPackage{normalem}{ulem}
\usepackage{ulem}

\makeatletter
\numberwithin{equation}{section}
\numberwithin{figure}{section}
\theoremstyle{plain}
\newtheorem{thm}{\protect\theoremname}[section]
  \theoremstyle{definition}
  \newtheorem{defn}[thm]{\protect\definitionname}
  \theoremstyle{definition}
  \newtheorem{example}[thm]{\protect\examplename}
  \theoremstyle{plain}
  \newtheorem{fact}[thm]{\protect\factname}
  \theoremstyle{remark}
  \newtheorem{rem}[thm]{\protect\remarkname}
  \theoremstyle{plain}
  \newtheorem{prop}[thm]{\protect\propositionname}
  \theoremstyle{remark}
  \newtheorem*{claim*}{\protect\claimname}
  \theoremstyle{definition}
  \newtheorem{problem}[thm]{\protect\problemname}
  \theoremstyle{plain}
  \newtheorem{cor}[thm]{\protect\corollaryname}
  \theoremstyle{plain}
  \newtheorem{lem}[thm]{\protect\lemmaname}

\usepackage{a4wide}
\usepackage{amssymb}
\usepackage{url}
\usepackage{ifthen}\usepackage{mathrsfs}\usepackage{pict2e}\usepackage{xargs}\usepackage{xspace}

\newcommandx{\intersection}[2][1 = undefined, 2 = undefined]{
  \ifthenelse{\equal{#1}{undefined}}{\cap}{
    \ifthenelse{\equal{#2}{undefined}}{\bigcap #1}{\bigcap_{#1} #2}
  }
}

\newcommandx{\union}[2][1 = undefined, 2 = undefined]{
  \ifthenelse{\equal{#1}{undefined}}{\cup}{
    \ifthenelse{\equal{#2}{undefined}}{\bigcup #1}{\bigcup_{#1} #2}
  }
}

\makeatother

\usepackage{babel}
  \providecommand{\claimname}{Claim}
  \providecommand{\corollaryname}{Corollary}
  \providecommand{\definitionname}{Definition}
  \providecommand{\examplename}{Example}
  \providecommand{\factname}{Fact}
  \providecommand{\lemmaname}{Lemma}
  \providecommand{\problemname}{Problem}
  \providecommand{\propositionname}{Proposition}
  \providecommand{\remarkname}{Remark}
\providecommand{\theoremname}{Theorem}

\begin{document}
\global\long\def\Aut{\operatorname{Aut}}
\global\long\def\calB{\mathscr{B}}
\global\long\def\calI{\mathcal{I}}
\global\long\def\ball#1#2{\calB(#1, #2)}
\global\long\def\calU{\mathcal{U}}
\global\long\def\calV{\mathcal{V}}
\global\long\def\Cantorspace{\functions{\omega}{2}}
\global\long\def\closure#1{\overline{#1}}
\global\long\def\concatenation{\mathrel{\smallfrown}}
\global\long\def\constantsequence#1#2{\left(#1\right){}^{#2}}
\global\long\def\diameter#1{\mathrm{diam}\left(#1\right)}
\global\long\def\Ezero{\mathbb{E}_{0}}
\global\long\def\from{\colon}
\global\long\def\Fsigma{F_{\sigma}}
\global\long\def\functions#1#2{#2^{#1}}
\global\long\def\heightcorrection#1{\raisebox{0pt}[0pt][0pt]{#1}}
\global\long\def\ideal#1{\calI_{#1}}
\global\long\def\image#1#2{#1\left[#2\right]}
\global\long\def\inverse#1{#1^{-1}}
\global\long\def\mathand{\text{ and }}
\global\long\def\N{\mathbb{N}}
\global\long\def\orbitequivalencerelation#1#2{E_{#1}^{#2}}
\global\long\def\pair#1#2{\left(#1,#2\right)}
\global\long\def\relationpower#1#2{#1^{\left(#2\right)}}
\global\long\def\restriction#1#2{#1 \upharpoonright#2}
\global\long\def\setcomplement#1{\twiddle#1}
\global\long\def\sets#1#2{\left[#2\right]{}^{#1}}
\global\long\def\suchthat{\mid}
\global\long\def\twiddle{\raisebox{1pt}{\scalebox{.75}{$\mathord{\sim}$}}}
\global\long\def\C{\mathfrak{C}}
\global\long\def\acl{\operatorname{acl}}
\global\long\def\tp{\operatorname{tp}}
\global\long\def\qf{\operatorname{qf}}
\global\long\def\Nn{\mathbb{N}}
\global\long\def\id{\operatorname{id}}
\global\long\def\SS{\mathcal{P}}
\global\long\def\EM{\operatorname{EM}}
\global\long\def\dcl{\operatorname{dcl}}
\global\long\def\Autf{\operatorname{Aut}f_{L}}
\global\long\def\eq{\operatorname{eq}}

\global\long\def\pamod#1{\pmod#1}

\global\long\def\nf{\mbox{nf}}
\global\long\def\Uu{\mathcal{U}}
\global\long\def\dom{\operatorname{dom}}
\global\long\def\concat{\smallfrown}
\global\long\def\Nn{\mathbb{N}}
\global\long\def\mathrela#1{\mathrel{#1}}
\global\long\def\twiddle{\mathord{\sim}}
\global\long\def\stab{\operatorname{stab}}
 \global\long\def\x{\times}
\global\long\def\diam{\operatorname{diam}}
\global\long\def\EZero{\mathbb{E}_{0}}
\global\long\def\sequence#1#2{\left\langle #1\left|\,#2\right.\right\rangle }
\global\long\def\set#1#2{\left\{  #1\left|\,#2\right.\right\}  }
\global\long\def\cardinal#1{\left|#1\right|}
\global\long\def\calO{\mathcal{O}}
\global\long\def\mathordi#1{\mathord{#1}}
\global\long\def\Ezero{\EZero}
\global\long\def\xx{\mathbf{x}}

\global\long\def\NTPT{\operatorname{NTP}_{\operatorname{2}}}
\global\long\def\ist{\operatorname{ist}}
\global\long\def\C{\mathfrak{C}}
\global\long\def\alt{\operatorname{alt}}
\global\long\def\Ff{\mathbb{F}}
\global\long\def\Ll{\mathfrak{L}}
\global\long\def\calU{\mathcal{U}}

\title{An embedding theorem of $\EZero$ with model theoretic applications}

\author{Itay Kaplan and Benjamin D. Miller}
\begin{abstract}
We provide a new criterion for embedding $\EZero$, and apply it to
equivalence relations in model theory. This generalize the results
of the authors and Pierre Simon on the Borel cardinality of Lascar
strong types equality, and Newelski's results about pseudo $F_{\sigma}$
groups. 
\end{abstract}

\thanks{The authors were supported in part by SFB Grant 878.}

\address{Itay Kaplan \\
The Hebrew University of Jerusalem\\
Einstein Institute of Mathematics \\
Edmond J. Safra Campus, Givat Ram\\
Jerusalem 91904, Israel}

\email{kaplan@math.huji.ac.il}

\urladdr{https://sites.google.com/site/itay80/ }

\address{Benjamin Miller \\
Institut f\"ur mathematische Logik und Grundlagenforschung \\
Fachbereich Mathematik und Informatik \\
Universit\"at M\"unster \\
Einsteinstra{\ss}e 62 \\
48149 M\"unster \\
Germany }

\email{ben.miller@uni-muenster.de}

\urladdr{http://wwwmath.uni-muenster.de/u/ben.miller/ }

\subjclass[2010]{03C45, 03E15}

\maketitle

\section{Introduction}

Given two topological spaces $X$ and $X'$ and two equivalence relations
$E$ and $E'$ respectively on $X$ and $X'$, we say that $E$ is
\emph{Borel reducible} to $E'$ if there is a Borel map $f$ from
$X$ to $X'$ such that $x\mathrela Ey\iff f(x)\mathrela{E'}f(y)$
for all $x,y\in X$. The quasi-order of Borel reducibility of Borel
equivalence relations on Polish spaces is a well-studied object in
descriptive set theory, and enjoys a number of remarkable properties.
One of them is given by the Harrington-Kechris-Louveau dichotomy,
which asserts that a Borel equivalence relation is either smooth (Borel
reducible to equality on $2^{\omega}$) or at least as complicated
as $\EZero$ (eventual equality on $2^{\omega}$). In other words,
$\EZero$ is the first non-smooth Borel equivalence relation. 

In Section \ref{sec:A-sufficient-condition}, we provide a new criterion
for being non-smooth. We also translate this criterion to another
context, that of strong Choquet spaces.

In the majority of Section \ref{sec:Applications}, we apply this
criterion to bounded invariant equivalence relation in model theory. 

Suppose $T$ is a complete first order theory and $\C$ a $\kappa$-saturated
model for some large $\kappa$. If $E$ is an equivalence relation
on $\C^{\alpha}$ which is a countable union of $\emptyset$-type
definable sets $U_{n}$ (i.e., $U_{n}$ is definable by intersection
of parameter free formulas), we say that it is bounded when the number
of classes is smaller than $\kappa$. We call $E$ a\emph{ bounded
invariant pseudo $F_{\sigma}$} equivalence relation. Such relations
appear naturally in model theory, and include the finest bounded invariant
equivalence relation: equality of Lascar strong types --- $\equiv_{L}^{\alpha}$.
It turns out that if the $T$ and $\alpha$ are countable, one can
interpret $E$ as an (honest) $F_{\sigma}$ equivalence relation on
a compact Polish space in a very natural way, which equips $E$ with
a well defined Borel cardinality. This was done for Lascar strong
types in \cite{PillayKrupinskiSolecki}, where many examples are computed,
but in fact works for any $E$. It is explained in details in Subsections
\ref{sub:Context} and \ref{sub:Countable-language}. 

If $E$ is an invariant bounded pseudo $F_{\sigma}$ equivalence relation,
we can assume by compactness that there are $\emptyset$-type definable
sets $U_{n}$ which are reflexive, symmetric and $U_{n}\circ U_{n}\subseteq U_{n+1}$
with $E=\bigcup_{n<\omega}U_{n}$. Such a sequence $\sequence{U_{n}}{n<\omega}$
is called a \emph{normal form} of $E$. 

In \cite[Corollary 1.12]{Newelski}, Newelski proved that if $E$
is an invariant pseudo $F_{\sigma}$ equivalence relation on $\C^{\alpha}$
with normal form $\sequence{U_{n}}{n<\omega}$, and $X$ is a type
definable set, all its elements have the same type over $\emptyset$,
then either $E\upharpoonright X=U_{n}$ for some $n$ or $\left|X/E\right|\geq2^{\aleph_{0}}$. 

({*}) Equivalently, if there is some $x\in X$ such that $E\upharpoonright\left[x\right]_{E}$
is not already $U_{n}\upharpoonright\left[x\right]_{E}$ for some
$n$ then $\left|X/E\right|\geq2^{\aleph_{0}}$. 

He continued to prove \cite[Theorem 3.1]{Newelski} that if $H$ is
an $\emptyset$-type definable group and $G\leq H$ is generated by
countably many sets $V_{n}$, each $\emptyset$-type definable, then
$G$ is type definable iff $G$ is generated by finitely many $V_{n}$'s
in finitely many steps and if $G$ is not type definable then $\left[H:G\right]\geq2^{\aleph_{0}}$.
In that case, if moreover $T$ is small (has only countably many types
over $\emptyset$) and $H$ consists of finite tuples, then $\left[H:G\right]$
is unbounded. Let $X=H$ and $E=E_{G}^{H}$ be the orbit equivalence
relation of the action of $G$ on $H$ (so it is an invariant pseudo
$F_{\sigma}$ equivalence relation).

({*}{*}) In this language this is equivalent to: if for some $x\in X$,
$E\upharpoonright\left[x\right]_{E}$ is not already generated by
finitely many of the $V_{n}$'s in finitely many steps, then $\left|X/E\right|\geq2^{\aleph_{0}}$. 

An important example of such a pair $\left(G,H\right)$ is $\left(G,G_{\emptyset}^{000}\right)$
where $G$ is $\emptyset$-type definable and $G_{\emptyset}^{000}$
is the minimal $\emptyset$-invariant bounded index subgroup. See
\cite{gismatulin} for more. 

In \cite{BorelCard} the authors dealt with the case where $X$ was
a $KP$-strong type and $E=\mathordi{\equiv_{L}^{\alpha}}$ The main
result there is that if $E$ is not trivial on $X$ then it is non-smooth.
This went through a stronger theorem \cite[Main Theorem A]{BorelCard}
that stated that:

({*}{*}{*}) If $Y$ is a pseudo $G_{\delta}$, $\equiv_{L}^{\alpha}$-invariant
subset of $\C^{\alpha}$ and for some $x\in Y$, $\left[x\right]_{\mathordi{\equiv_{L}^{\alpha}}}$
has unbounded Lascar diameter, then $\mathordi{\equiv_{L}^{\alpha}}\upharpoonright Y$
is non-smooth. ``Unbounded Lascar diameter'' means exactly that
it is not the case that $\mathordi{\equiv_{L}^{\alpha}}\upharpoonright\left[x\right]_{\mathordi{\equiv_{L}^{\alpha}}}=U_{n}\upharpoonright\left[x\right]_{\mathordi{\equiv_{L}^{\alpha}}}$
for some $n$, where $U_{n}\left(a,b\right)$ is the type saying that
the Lascar distance between $a$ and $b$ is at most $n$. (This is
a normal form for $\equiv_{L}^{\alpha}$.)

Here we try to generalize ({*}), ({*}{*}) and ({*}{*}{*}) in a uniform
way using the results from Section \ref{sec:A-sufficient-condition}.
So the idea is to prove, in each case (when everything is countable),
that if $Y$ is a pseudo $G_{\delta}$, $E$-invariant and for some
$x\in Y$, $E\upharpoonright\left[x\right]_{E}$ is not already $U_{n}\upharpoonright\left[x\right]_{E}$
for some $n$, then $E\upharpoonright X$ is not smooth. 

While we do not successfully generalize ({*}), we do prove it if there
is a subgroup of $\Aut\left(\C\right)$ which acts nicely on $\left[x\right]_{E}$,
for instance when it is transitive on this class and preserves all
classes. This is done in Subsection \ref{sub:First-application:-invariant},
and includes also ({*}{*}{*}) (the subgroup in that case is $\Autf\left(\C\right)$).
({*}{*}) is successfully generalized and moreover stated for group
actions (with an extra technical assumption called ``shiftiness''
which holds in the case where the action is free). 

\medskip{}

We would like to thank Ziv Shami and Pierre Simon for some useful
conversations.

\section{\label{sec:A-sufficient-condition}A sufficient condition for embedding
$\Ezero$}

\subsection{\label{sub:Preliminaries}Preliminaries}
\begin{defn}
Suppose $X$ and $Y$ are topological spaces, and $E$ and $F$ are
Borel equivalence relations on $X$ and $Y$. We say that a function
$f:X\to Y$ is a \emph{reduction} of $E$ to $F$ if for all $x_{0},x_{1}\in X$,
$\left(x_{0},x_{1}\right)\in E$ iff $\left(f\left(x_{0}\right),f\left(x_{1}\right)\right)\in F$.
\begin{enumerate}
\item We say that $E$ is\emph{ Borel reducible} to $F$, denoted by $E\leq_{B}F$,
when there is a Borel reduction $f:X\to Y$ of $E$ to $F$.
\item We write $E\sqsubseteq_{c}F$ when there is a continuous injective
reduction $f:X\to Y$ of $E$ to $F$.
\item We say that $E$ and $F$ are \emph{Borel bi-reducible}, denoted by
$E\sim_{B}F$, when $E\leq_{B}F$ and $F\leq_{B}E$. 
\item We write $E<_{B}F$ to mean that $E\leq_{B}F$ but $E\not\sim_{B}F$. 
\end{enumerate}
\end{defn}
\begin{example}
For a Polish space $X$, the relations $\Delta\left(X\right)$ denotes
equality on $X$. Then $\Delta\left(1\right)<_{B}\Delta\left(2\right)<_{B}\ldots<_{B}\Delta\left(\omega\right)<_{B}\Delta\left(2^{\omega}\right)$. \end{example}
\begin{defn}
We say that $E$ is \emph{smooth} iff $E\leq_{B}\Delta\left(2^{\omega}\right)$. \end{defn}
\begin{fact}
\cite{silver} (Silver dichotomy) For all Borel equivalence relations
$E$ on Polish spaces, $E\leq_{B}\Delta\left(\omega\right)$ or $\Delta\left(2^{\omega}\right)\sqsubseteq_{c}E$
. It follows that $\Delta\left(2^{\omega}\right)$ is the successor
of $\Delta\left(\omega\right)$.
\end{fact}

\begin{fact}
\label{fac:Closed-equivalence-relations}Closed equivalence relations
are smooth. \end{fact}
\begin{example}
Let $\EZero$ be the following equivalence relation on the Cantor
space $2^{\omega}$: $\left(\eta,\nu\right)\in\EZero$ iff there exists
some $n<\omega$ such that for all $m>n$, $\eta\left(m\right)=\nu\left(m\right)$. \end{example}
\begin{fact}
\label{fac:E_0 is not definable smooth}The relation $\EZero$ is
non-smooth.
\end{fact}
In addition, we have the following dichotomy:
\begin{fact}
\label{fac:Harrington-Kechris-Louveau-dich} \cite{harringtonKechrisLouveau}
(Harrington-Kechris-Louveau dichotomy) For every Borel equivalence
relation $E$ on a Polish space either $E\leq_{B}\Delta\left(2^{\omega}\right)$
(i.e., $E$ is smooth) or $\EZero\sqsubseteq_{c}E$. It follows that
$\EZero$ is the successor of $\Delta\left(2^{\omega}\right)$. 
\end{fact}

\subsection{The ideal embedding theorem}

Suppose that $X$ is a topological space. Associated with each family
$\calU$ of open subsets of $X$ is the corresponding family $\ideal{\calU}$
of subsets of $X$ given by $F\in\ideal{\calU}\iff\forall U\in\calU\exists\mbox{ open }V\supseteq F\ V\union U\in\calU$.

Equivalently,
\begin{rem}
\label{rem:equivalent definition} $F\in\ideal{\calU}\iff\forall U\in\calU\exists V\in\calU\ F\union U\subseteq V$. \end{rem}
\begin{prop}
\label{proposition:ideal} Suppose that $X$ is a topological space
and $\calU$ is a family of open subsets of $X$. Then $\ideal{\calU}$
is an ideal. \end{prop}
\begin{proof}
To see that $\ideal{\calU}$ is downward closed, note that if $F\in\ideal{\calU}$
and $F'\subseteq F$, then for each $U\in\calU$, there exists an
open set $V\supseteq F\supseteq F'$ with the property that $V\union U\in\calU$,
thus $F'\in\ideal{\calU}$.

To see that $\ideal{\calU}$ is closed under finite unions, note that
if $F,F'\in\ideal{\calU}$, then for each $U\in\calU$, there exists
an open set $V\supseteq F$ with $V\union U\in\calU$, so there exists
an open set $V'\supseteq F'$ with $\left(V\union V'\right)\union U\in\calU$,
thus $F\union F'\in\ideal{\calU}$.\end{proof}
\begin{prop}
\label{proposition:splitting} Suppose that $X$ is a topological
space, $\Gamma$ is a group of homeomorphisms of $X$, $Y\subseteq X$
is $\Gamma$-invariant, and $\calU$ is the family of open sets $U\subseteq X$
for which there is no finite set $\Delta\subseteq\Gamma$ with $Y\subseteq\Delta\cdot U$.
Then for all $\ideal{\calU}$-positive sets $F\subseteq X$ and all
open sets $W\supseteq F$, there is a finite set $\Delta\subseteq\Gamma$
such that whenever $I$ is a finite set, $\sequence{F_{i}}{i\in I}$
is a finite sequence of subsets of $X$ whose union contains $F$,
and $\sequence{\lambda_{i}}{i\in I}$ is a sequence of elements of
$\Gamma$, there exists $\delta\in\Delta$ and $i\in I$ for which
$\closure{\delta\cdot W}\intersection\lambda_{i}\cdot F_{i}$ is $\ideal{\calU}$-positive. \end{prop}
\begin{proof}
We will use Remark \ref{rem:equivalent definition}. Fix $U\in\calU$
such that for no $V\in\calU$ is $F\union U\subseteq V$. Then there
is a finite set $\Delta\subseteq\Gamma$ with $Y\subseteq\Delta\cdot(U\union W)$.
In light of Proposition \ref{proposition:ideal}, it is sufficient
to show that there is no finite set $I$, sequence $\sequence{F_{i}}{i\in I}$
of subsets of $X$ whose union contains $F$, and sequence $\sequence{\lambda_{i}}{i\in I}$
such that $\closure{\Delta\cdot W}\intersection\union[i\in I][\lambda_{i}\cdot F_{i}]\in\ideal{\calU}$.

Suppose, towards a contradiction, that there is such a triple. Then
$Y\setminus\closure{\Delta\cdot W}\subseteq\Delta\cdot U$, so $\setcomplement{\closure{\Delta\cdot W}}\union\union[i\in I][\lambda_{i}\cdot U]\in\calU$.
Fix $V\in\calU$ with $\union[i\in I][\lambda_{i}\cdot(U\union F_{i})]\subseteq(\closure{\Delta\cdot W}\intersection\union[i\in I][\lambda_{i}\cdot F_{i}])\union(\setcomplement{\closure{\Delta\cdot W}}\union\union[i\in I][\lambda_{i}\cdot U])\subseteq V$.
Then $\union[i\in I][\inverse{\lambda_{i}}\cdot V]\in\calU$ and $F\union U\subseteq\union[i\in I][\inverse{\lambda_{i}}\cdot V]$,
a contradiction.\end{proof}
\begin{thm}
\label{thm:embedding1}Suppose that $X$ is a complete metric space,
$\Gamma$ is a group of homeomorphisms of $X$, $Y\subseteq X$ is
$\Gamma$-invariant, $\calU$ is the family of open sets $U\subseteq X$
for which there is no finite set $\Delta\subseteq\Gamma$ such that
$Y\subseteq\Delta\cdot U$, $\sequence{R_{n}}{n\in\N}$ is an increasing
sequence of reflexive symmetric closed subsets of $X\times X$, and
there is a compact $\ideal{\calU}$-positive set $K\subseteq X$ with
the following properties:  \renewcommand{\theenumi}{\alph{enumi}} 
\begin{enumerate}
\item $\forall n\in\N\forall x\in K\exists\gamma\in\Gamma\ \neg x\mathrel{\relationpower{R_{n}}4}\gamma\cdot x$. 
\item $\forall\gamma\in\Gamma\exists n\in\N\forall x\in\Gamma\cdot K\ x\mathrel{R_{n}}\gamma\cdot x$. 
\end{enumerate}
Then for some $x\in K$ there is a continuous injective homomorphism
$\phi\from\Cantorspace\to\closure{\Gamma\cdot x}$ from $\pair{\Ezero}{\setcomplement{\Ezero}}$
into $\pair{\orbitequivalencerelation{\Gamma}X}{\setcomplement{\union[n\in\N][R_{n}]}}$. \end{thm}
\begin{proof}
Let $\calV$ denote the family of open sets $V\subseteq X$ containing
compact $\ideal{\calU}$-positive subsets of $K$. We recursively
construct $V_{n}\in\calV$ and $\gamma_{n}\in\Gamma$, from which
we define $\gamma_{s}=\prod_{i<n}\gamma_{i}^{s(i)}$ for $s\in\functions{<\omega}2$,
so as to ensure that at stage $n$ of the construction, the following
conditions are satisfied: 
\begin{enumerate}
\item $\forall m<n\ \closure{V_{m+1}}\union\gamma_{m}\cdot\closure{V_{m+1}}\subseteq V_{m}$. 
\item $\forall m<n\forall s\in\functions{m+1}2\ \diameter{\gamma_{s}\cdot V_{m+1}}\le1/m$. 
\item $\forall m<n\forall s,t\in\functions m2\ \left(\left(\gamma_{s}\cdot V_{m+1}\right)\times\left(\gamma_{t}\gamma_{m}\cdot V_{m+1}\right)\right)\intersection R_{m}=\emptyset$. 
\end{enumerate}
We begin by setting $V_{0}=X$.

Suppose now that $n\in\N$ and we have found $V_{n}$ and $\sequence{\gamma_{i}}{i<n}$.
Fix an $\ideal{\calU}$-positive compact set $L\subseteq K$ contained
in $V_{n}$, as well as an open set $W\supseteq L$ for which $\closure W\subseteq V_{n}$.
By Proposition \ref{proposition:splitting}, there is a finite set
$\Delta\subseteq\Gamma$ such that whenever $I$ is a finite set,
$\sequence{L_{i}}{i\in I}$ is a sequence of compact sets whose union
is $L$, and $\sequence{\lambda_{i}}{i\in I}$ is a sequence of elements
of $\Gamma$, there are $\delta\in\Delta$ and $i\in I$ for which
$\closure{\delta\cdot W}\intersection\lambda_{i}\cdot L_{i}$ is $\ideal{\calU}$-positive.
Condition (b) yields $m\ge n$ such that $\forall x\in\Gamma\cdot K\forall\gamma\in\set{\gamma_{s}}{s\in\functions n2}\union\Delta^{-1}\ x\mathrel{R_{m}}\gamma\cdot x$.
In particular, it follows that:
\begin{itemize}
\item [(*)]If $x\in K$, $\lambda\in\Gamma$ and $\neg x\mathrel{\relationpower{R_{m}}4}\lambda\cdot x$,
then for no $\delta\in\Delta$ and $s,t\in\functions n2$ is it the
case that $\gamma_{s}\cdot x\mathrel{R_{m}}\gamma_{t}\inverse{\delta}\lambda\cdot x$.
\end{itemize}
Thus condition (a) yields a finite set $I$, a sequence $\sequence{L_{i}}{i\in I}$
of compact subsets of $X$ whose union is $L$, and a sequence $\sequence{\lambda_{i}}{i\in I}$
of elements of $\Gamma$ with 
\[
\forall\delta\in\Delta\forall i\in I\forall s,t\in\functions n2\ (\gamma_{s}\cdot L_{i}\times\gamma_{t}\inverse{\delta}\lambda_{i}\cdot L_{i})\intersection R_{m}=\emptyset.
\]
Fix $\delta\in\Delta$ and $i\in I$ such that $\closure{\inverse{\lambda_{i}}\delta\cdot W}\intersection L_{i}$
is $\ideal{\calU}$-positive, and define $\gamma_{n}=\inverse{\delta}\lambda_{i}$.
Then $\forall s,t\in\functions n2\ ((\gamma_{s}\cdot L_{i})\times(\gamma_{t}\gamma_{n}\cdot L_{i}))\intersection R_{m}=\emptyset$.
Proposition \ref{proposition:ideal} ensures that by replacing $L_{i}$
with a compact $\ideal{\calU}$-positive subset of $\closure{\inverse{\lambda_{i}}\delta\cdot W}\intersection L_{i}$,
we can assume that $\forall s\in\functions{n+1}2\ \diameter{\gamma_{s}\cdot L_{i}}<1/n$.
It follows that there is an open set $V_{n+1}\subseteq X$ containing
$L_{i}$ such that $\closure{V_{n+1}}\union(\gamma_{n}\cdot\closure{V_{n+1}})\subseteq V_{n}$,
$\forall s\in\functions{n+1}2\ \diameter{\gamma_{s}\cdot V_{n+1}}\le1/n$,
and $\forall s,t\in\functions n2\ ((\gamma_{s}\cdot V_{n+1})\times(\gamma_{t}\gamma_{n}\cdot V_{n+1}))\intersection R_{n}=\emptyset$.
This completes the recursive construction.

Conditions (1) and (2) ensure that we obtain a continuous function
$\phi\from\Cantorspace\to X$ by insisting that $\left\{ \phi\left(c\right)\right\} =\intersection[n\in\N][\gamma_{\restriction cn}\cdot V_{n}]$
for all $c\in\Cantorspace$. To see that $\phi$ is a homomorphism
from $\Ezero$ to $\orbitequivalencerelation{\Gamma}X$, suppose that
$c\in\Cantorspace$, $k\in\N$, and $s\in\functions k2$, and observe
that 
\[
\left\{ \gamma_{s}\cdot\phi(\constantsequence 0k\concatenation c)\right\} ={\textstyle \intersection[n\in\N][\gamma_{s}\gamma_{\restriction{(\constantsequence 0k\concatenation c)}n}\cdot V_{n}]=\left\{ \phi(s\concatenation c)\right\} .}
\]
To see that $\phi$ is an injective homomorphism from $\setcomplement{\Ezero}$
to $\setcomplement{\union[n\in\N][R_{n}]}$, suppose that $c,d\in\Cantorspace$,
$n\in\N$, $c(n)=0$, and $d(n)=1$, and observe that $\phi(c)\in\gamma_{\restriction cn}\cdot V_{n+1}$
and $\phi(d)\in\gamma_{\restriction dn}\gamma_{n}\cdot V_{n+1}$,
in which case condition (3) ensures that $\neg\phi(c)\mathrel{R_{n}}\phi(d)$.
Finally, set $x=\phi(\constantsequence 0{\infty})$ and note that
$\image{\phi}{\Cantorspace}\subseteq\closure{\Gamma\cdot x}$ and
$x\in K$. 
\end{proof}
We give a slight variant of Theorem \ref{thm:embedding1}, adding
an extra assumption. 
\begin{thm}
\label{thm:embedding2}Suppose that $X$, $\Gamma$, $Y$, $\calU$,
$\sequence{R_{n}}{n\in\N}$ are as in Theorem \ref{thm:embedding1}.
Suppose that there is a compact $\ideal{\calU}$-positive set $K\subseteq X$
with the following properties:  \renewcommand{\theenumi}{\alph{enumi}} 
\begin{enumerate}
\item $\forall n\in\N\forall x\in K\exists\gamma\in\Gamma\ \neg x\mathrel{\relationpower{R_{n}}2}\gamma\cdot x$. 
\item $\forall\gamma\in\Gamma\exists n\in\N\forall x\in K\ x\mathrel{R_{n}}\gamma\cdot x$.
Note that this condition is weaker than (b) in Theorem \ref{thm:embedding1}.
\item $\forall\gamma\in\Gamma\forall x,y\in\Gamma\cdot K\forall n\in\Nn\ x\mathrela{R_{n}}y\Rightarrow\gamma\cdot x\mathrela{R_{n}}\gamma\cdot y$.
\end{enumerate}

Then for some $x\in K$ there is a continuous injective homomorphism
$\phi\from\Cantorspace\to\closure{\Gamma\cdot x}$ from $\pair{\Ezero}{\setcomplement{\Ezero}}$
into $\pair{\orbitequivalencerelation{\Gamma}X}{\setcomplement{\union[n\in\N][R_{n}]}}$. 

\end{thm}
\begin{proof}
The proof is parallel to the proof of Theorem \ref{thm:embedding1},
reading the same up to ({*}), but we choose $m$ so that $\forall x\in K\forall\gamma_{1},\gamma_{2}\in\set{\gamma_{s}}{s\in\functions n2}\forall\delta\in\Delta\ x\mathrel{R_{m}}\delta\gamma_{1}^{-1}\gamma_{2}\cdot x$.
By (c), we get:
\begin{itemize}
\item [(**)]If $x\in K$, $\lambda\in\Gamma$ and $\neg x\mathrel{\relationpower{R_{m}}2}\lambda\cdot x$,
then for no $\delta\in\Delta$ and $s,t\in\functions n2$ is it the
case that $\gamma_{s}\cdot x\mathrel{R_{m}}\gamma_{t}\inverse{\delta}\lambda\cdot x$.
\end{itemize}
The rest of the proof is exactly the same. 
\end{proof}

\subsection{Choquet spaces}

The proof of Theorem \ref{thm:embedding1} easily goes through in
the context of strong Choquet spaces. 
\begin{defn}
The \emph{Choquet game} on a topological space $X$ is a two player
game in $\omega$-many rounds. In round $n$, player A chooses a nonempty
open set $U_{n}\subseteq V_{n-1}$ (where $V_{-1}=X$), and player
B responds by choosing a nonempty open subset $V_{n}\subseteq U_{n}$.
Player B wins if the intersection $\bigcap\set{V_{n}}{n<\omega}$
is not empty. 

The \emph{strong Choquet game }is similar: in round $n$ player A
chooses an open set $U_{n}\subseteq V_{n-1}$ and $x_{n}\in U_{n}$,
and player B responds by choosing an open set $V_{n}\subseteq U_{n}$
containing $x_{n}$. Again, player B wins when the intersection $\bigcap\set{V_{n}}{n<\omega}$
is not empty. 

A topological space $X$ is a \emph{(strong) Choquet space} if player
B has a winning strategy in every (strong) Choquet game. 
\end{defn}
It is easy to see that:
\begin{example}
Every Polish space is strong Choquet. 
\end{example}
But for our purposes, we shall need the following example:
\begin{example}
If $X$ is compact (not necessarily Hausdorff) and has a basis consisting
of clopen sets then it is strong Choquet.\end{example}
\begin{proof}
In round $n$, player B will choose a clopen set $x_{n}\in V_{n}\subseteq U_{n}$.
By compactness, the intersection $\bigcap\set{V_{n}}{n<\omega}$ is
not empty. \end{proof}
\begin{fact}
\label{fac:G_delta subset is Choquet} (see e.g., \cite{BorelCard})
If $X$ is strong Choquet and $\emptyset\neq U\subseteq X$ is $G_{\delta}$,
then $U$ is also strong Choquet.\end{fact}
\begin{thm}
\label{thm:embedding Choquet1} Suppose that $X$ is a regular strong
Choquet space, $\Gamma$ is a group of homeomorphisms of $X$, $Y\subseteq X$
is $\Gamma$-invariant, $\calU$ is the family of open sets $U\subseteq X$
for which there is no finite set $\Delta\subseteq\Gamma$ such that
$Y\subseteq\Delta\cdot U$, $\sequence{R_{n}}{n\in\N}$ is an increasing
sequence of reflexive symmetric closed subsets of $X\times X$, and
there is a compact $\ideal{\calU}$-positive set $K\subseteq X$ with
the following properties:  \renewcommand{\theenumi}{\alph{enumi}} 
\begin{enumerate}
\item $\forall n\in\N\forall x\in K\exists\gamma\in\Gamma\ \neg x\mathrel{\relationpower{R_{n}}4}\gamma\cdot x$. 
\item $\forall\gamma\in\Gamma\exists n\in\N\forall x\in\Gamma\cdot K\ x\mathrel{R_{n}}\gamma\cdot x$. 
\end{enumerate}
Then there is a map $\phi:2^{\omega}\to\SS\left(X\right)$ such that
for every $y,z\in2^{\omega}$: 
\begin{itemize}
\item $\phi\left(y\right)$ is a nonempty closed $G_{\delta}$ subset of
$X$.
\item If $z\mathrela{\EZero}y$ then there is some $\gamma\in\Gamma$ such
that $\gamma\cdot\phi\left(z\right)=\phi\left(y\right)$.
\item If $\twiddle z\mathrela{\EZero}y$ then $\left(\mbox{\ensuremath{\phi}}\left(y\right)\times\mbox{\ensuremath{\phi}}\left(z\right)\right)\subseteq\setcomplement{\union[n\in\N][R_{n}]}$.
\end{itemize}
In particular, there is homomorphism $\phi\from\Cantorspace\to X$
from $\pair{\Ezero}{\setcomplement{\Ezero}}$ into $\pair{\orbitequivalencerelation{\Gamma}X}{\setcomplement{\union[n\in\N][R_{n}]}}$. 

Moreover, if $X$ is compact, then we can choose $\phi$ so that its
range is $\closure{\Gamma\cdot x}$ for some $x\in K$.\end{thm}
\begin{proof}
Fix a winning strategy for Player B in the strong Choquet game. The
main point is that in the construction done in the proof of Theorem
\ref{thm:embedding1}, instead of decreasing the diameter of the open
sets, we choose them according to the strategy. So in addition to
choosing $V_{n}$ and $\gamma_{n}$, we also choose points $x_{n}\in X$
and open sets $U_{n}\in\calV$ (the family of open sets containing
compact $\ideal{\calU}$-positive subsets of $K$) such that $x_{n}\in U_{n}\subseteq V_{n}$,
and the new construction will satisfy:
\begin{enumerate}
\item $\forall m<n\ \closure{U_{m}}\union\gamma_{m}\cdot\closure{U_{m}}\subseteq V_{m}$. 
\item $\forall m<n\forall s,t\in\functions m2\ \left(\left(\gamma_{s}\cdot U_{m}\right)\times\left(\gamma_{t}\gamma_{m}\cdot U_{m}\right)\right)\intersection R_{m}=\emptyset$. 
\item $\forall m<n$ and $\forall s\in\functions{m+1}2$, $\gamma_{s}\cdot V_{m+1}$
is contained in an open set which is played according to Player B's
strategy in the Strong Choquet game in which Player A plays $\sequence{\left(\gamma_{s\upharpoonright i+1}\cdot U_{i},\gamma_{s\upharpoonright i+1}\cdot x_{i}\right)}{i\leq m}$
and Player B plays according to his strategy.
\end{enumerate}
For the construction, we follow the proof of Theorem \ref{thm:embedding1},
and note that:
\begin{claim*}
Suppose $L$ is an $I_{\Uu}$-positive compact set contained in some
open set $U$, and suppose $\Delta$ is a finite subset of $\Gamma$.
Furthermore, suppose that for any $\gamma\in\Delta$, $\gamma\cdot U$
is contained in an open set which is chosen by Player B in some finite
strong Choquet play according to his strategy. Then, there is some
$x\in L$ such that if Player A plays $\left(\gamma\cdot U,\gamma\cdot x\right)$
then there is a set $V\in\calV$ contained in $U$ such that $\gamma\cdot V$
is contained in Player B's response for all $\gamma\in\Delta$. \end{claim*}
\begin{proof}
Indeed, for each point $x\in L$, let Player A play $\left(\gamma\cdot U,\gamma\cdot x\right)$,
and let $\gamma\cdot U_{\gamma,x}$ be Player B's response. Let $U_{x}=\bigcap_{\gamma\in\Delta}U_{x,\gamma}$,
and let $U'_{x}$ be such that $x\in U_{x}'$ and $\closure{U_{x}'}\subseteq U{}_{x}$.
By compactness and by Proposition \ref{proposition:ideal}, for some
$x\in L$, $\closure{U'_{x}}\cap L$ is $I_{\Uu}$-positive. Let $V=U_{x}$. 
\end{proof}
Now we let $U_{n}$ be the set denoted $V_{n+1}$ in the proof of
Theorem \ref{thm:embedding1} (without the condition on the diameter),
and proceed using the claim. Finally, we let $\phi\left(c\right)=\intersection[n\in\N][\gamma_{\restriction cn}\cdot V_{n}]$.

For the moreover part, choose any $x\in K\cap\phi\left(\constantsequence 0{\omega}\right)$,
and note that by compactness $\closure{\Gamma\cdot x}\cap\phi\left(c\right)\neq\emptyset$. 
\end{proof}
We also have an analog to Theorem \ref{thm:embedding2}, which we
state briefly. 
\begin{thm}
\label{thm:embedding Choquet2}Suppose that $X$, $\Gamma$, $Y$,
$\calU$, $\sequence{R_{n}}{n\in\N}$ are as in Theorem \ref{thm:embedding Choquet1}.
Suppose that there is a compact $\ideal{\calU}$-positive set $K\subseteq X$
satisfying the assumptions of Theorem \ref{thm:embedding2}.

Then the conclusion of Theorem \ref{thm:embedding Choquet1} hold. \end{thm}
\begin{problem}
All the applications we found use a weak version of Theorem \ref{thm:embedding1},
i.e., we apply it with (a) replaced by: $\forall n\in\N\exists\gamma\in\Gamma\forall x\in K\ \neg x\mathrel{\relationpower{R_{n}}4}\gamma\cdot x$.
Is there an interesting application that uses the full strength of
the theorem?
\end{problem}

\section{\label{sec:Applications}Applications}

\subsection{Application to compact group actions}

Most of our applications will be model theoretic, but we start with
a simple topological one. 
\begin{cor}
Suppose that $G$ is a compact topological group and that $\cdot:G\x X\to X$
is a continuous action on a complete metric space $X$. Let $H$ a
subgroup of $G$, and suppose $H=\bigcup_{n<\omega}V_{n}$ where $V_{n}$
are closed subsets of $G$, $e\in V_{n}^{-1}=V_{n}$, $V_{n}^{2}\subseteq V_{n+1}$.
Then, if there is some $x\in X$ such that $H\cdot x\neq V_{n}\cdot x$
for all $n<\omega$, then $\EZero\sqsubseteq_{c}E_{H}^{X}$. If not
and $X$ is Polish then $E_{H}^{X}$ is smooth. \end{cor}
\begin{proof}
First assume that there is such an $x\in X$. Let $Y=H\cdot x$, $R_{n}=\set{\left(x,y\right)\in X^{2}}{\exists h\in V_{n}\left(y=h\cdot x\right)}$,
$\Gamma=H$ and $K=\left\{ x\right\} $. All the conditions of Theorem
\ref{thm:embedding1} but the condition that $K$ is $I_{\Uu}$-positive
hold trivially (note that $R_{n}$ is closed by the compactness of
$V_{n}$). To show that $K$ is $I_{\Uu}$-positive it is enough to
see that for any open $x\in V\subseteq X$, there is some finite $\Delta\subseteq H$
such that $\Delta\cdot V\supseteq Y$. Suppose not. Recursively choose
$h_{n}\in H$ for $n<\omega$ such that $h_{n}\cdot x\notin\bigcup_{i<n}h_{i}\cdot V$. 

Let $\kappa=\left|G\right|^{+}$, and let $L=G^{\kappa}$ equipped
with the product topology (so it is compact). For a finite $s\subseteq\kappa$,
let $F_{s}=\set{\eta\in L}{\forall i\in s\left(\eta\left(i\right)\cdot x\notin\bigcup_{j<i,j\in s}\eta\left(j\right)\cdot V\right)}$.
By the construction above, this is a closed nonempty set. By compactness,
there is some $\eta\in\bigcap_{s\subseteq\kappa,\left|s\right|<\omega}F_{s}$.
In particular $\eta:\kappa\to G$ is injective --- contradiction. 

Now assume that $X$ is Polish and that there is no such $x$ but
$E_{H}^{X}$ is not smooth. By assumption, for all $x\in X$, $H\cdot x=V_{n}\cdot x$
for some $n<\omega$ and as $G$ is compact it follows that all classes
are compact, so also $G_{\delta}$. But a Borel equivalence relation
$E$ with $G_{\delta}$ classes on a Polish space $X$ must be smooth.
Otherwise Fact \ref{fac:Harrington-Kechris-Louveau-dich} gives us
a continuous embedding of $\EZero$ into $E$, and it follows that
every $\EZero$ class is $G_{\delta}$. But $\EZero$ classes are
also dense --- contradiction. \end{proof}
\begin{cor}
If $G$ is a compact complete metric group, and $H$ is an $F_{\sigma}$
subgroup, then either $H$ is closed (in which case, if $G$ is Polish,
$E_{H}^{G}$ is smooth), or $\EZero\sqsubseteq_{c}E_{H}^{G}$. 
\end{cor}

\subsection{\label{sub:preliminaries}Applications to model theory}

In applying Theorem \ref{thm:embedding1} or any of its variations,
we need to find the space $X$, the set $Y$, the group $\Gamma$,
the closed sets $R_{n}$ and the compact $I_{\Uu}$-positive set $K$.
In all our applications, $X$ will be some subspace of $S_{\alpha}\left(M\right)$
for model $M$, invariant under $E$, $Y$ will be the projection
of some $E$-class $C$, $R_{n}$ will be the projections of $U_{n}$
(from $E$'s normal form), $\Gamma$ will be some group of homeomorphisms
of $X$, which is either induced by automorphisms of the model $M$
or by a type definable model theoretic group and $K$ will be the
projection of some type of the form $U_{n}\left(x,a\right)$. The
main point is to show that $K$ is $I_{\Uu}$-positive, which we will
call here ``proper''.

\subsubsection{Preliminaries}

We briefly introduce our notation, which is fully explained in \cite{BorelCard}.
\begin{itemize}
\item $T$ is a complete (perhaps many sorted) first order theory.
\item $\alpha$ is some ordinal.
\item $S_{\alpha}\left(A\right)$ is the Stone space of complete $\alpha$-types
over $A$, which comes equipped with a compact Hausdorff topology,
and $L_{\alpha}\left(A\right)$ is the set of formulas in the first
$\alpha$ variables. 
\item $\C$ is a monster model of $T$ --- a $\kappa$-saturated, $\kappa$-homogeneous
model where $\kappa$ is a big cardinal.
\item All parameter sets and models considered will be \emph{small} (i.e.,
of cardinality less than $\kappa$) subsets and elementary substructures
of $\C$.
\item $\equiv$ is equality of types, $\mathordi{\equiv_{L}^{\alpha}}$
is equality of Lascar strong types of $\alpha$-tuples (if $A$ is
a small set, then $\equiv_{A}$ denotes types equality over $A$,
etc.).
\item $\Aut\left(\C/A\right)$ is the group of automorphisms of $\C$ that
fix $A$ pointwise, and an $A$-invariant subset of $\C^{\alpha}$
is one invariant under the action of this group.
\item A subset $X\subseteq\C^{\alpha}$ is \emph{pseudo closed} if $X$
is type definable over some small set. A \emph{pseudo open} set is
a complement of a pseudo closed set. \emph{Pseudo $G_{\delta}$} sets
and \emph{pseudo $F_{\sigma}$} sets are defined in the obvious way. 
\item If $Y\subseteq\C^{\alpha}$ is some set, and $M$ some model then
$Y_{M}=\set{p\in S_{\alpha}\left(M\right)}{\exists a\in Y\left(p=\tp\left(a/M\right)\right)}$.
This is also denoted by $S_{M}\left(Y\right)$. 
\end{itemize}
We also recall the notion of an indiscernible sequence:
\begin{defn}
Let $A$ be a small set. Let $\left(I,<\right)$ be some linearly
ordered set. A sequence $\bar{a}=\sequence{a_{i}}{i\in I}\in\left(\C^{\alpha}\right)^{I}$
is called \emph{$A$-indiscernible} (or \emph{indiscernible over $A$})
if for all $n<\omega$, every increasing $n$-tuple from $\bar{a}$
realizes the same type over $A$. When $A$ is omitted, it is understood
that $A=\emptyset$. 
\end{defn}
Also recall:
\begin{fact}
\label{fac:indiscernibles exist}$ $
\begin{enumerate}
\item \cite[Lemma 5.1.3]{TentZiegler} Let $\left(I,<_{I}\right)$, $\left(J,<_{J}\right)$
be small linearly ordered sets, and let $A$ be some small set. Suppose
$\bar{b}=\sequence{b_{j}}{j\in J}$ is some sequence of elements from
$\C^{\alpha}$. Then there exists an indiscernible sequence $\bar{a}=\sequence{a_{i}}{i\in I}\in\left(\C^{\alpha}\right)^{I}$
such that:

\begin{itemize}
\item For any $n<\omega$ and $\varphi\in L_{\alpha\cdot n}$, if $\C\models\varphi\left(b_{j_{0}},\ldots,b_{j_{n-1}}\right)$
for every $j_{0}<_{J}\ldots<_{J}j_{n-1}$ from $J$ then $\C\models\varphi\left(a_{i_{0}},\ldots,a_{i_{n-1}}\right)$
for every $i_{0}<_{I}\ldots<_{I}i_{n-1}$ from $I$.
\end{itemize}
\item \cite[proof of Proposition 3.1.4]{AvivThesis} If $M$ is a small
model and $a\equiv_{M}b$, then there is an indiscernible sequence
$\bar{c}=\sequence{c_{i}}{i<\omega}$ such that both $a\concat\bar{c}$
and $b\concat\bar{c}$ are indiscernible. 
\end{enumerate}
\end{fact}

\subsubsection{\label{sub:Context}Equivalence relations on $\C^{\alpha}$}
\begin{defn}
An equivalence relation $E$ on a set $X$ is called \emph{bounded}
if $\left|X/E\right|<\kappa$. 
\end{defn}
See \cite[Remark 1.12]{BorelCard} for a discussion of bounded invariant
equivalence relations. 

Suppose that $A$ is some small set, $X\subseteq\C^{\alpha}$ is type
definable over $A$, and that $E$ is some $\emptyset$-invariant
relation on $\C^{\alpha\cdot2}$ such that $E\upharpoonright X$ is
a bounded equivalence relation on $X$. 
\begin{defn}
Let $M\supseteq A$ be any model. For $p,q\in S_{X}\left(M\right)$,
we write $p\mathrela{E^{M}}q$ iff $\exists a\models p,b\models q\,\left(a\mathrela Eb\right)$.
\end{defn}
Note that this does not depend on the choice of representatives, i.e.,:
\begin{prop}
\label{prop:For all =00003D exists}For $p,q\in S_{X}\left(M\right)$,
$p\mathrela{E^{M}}q$ iff $\forall a\models p,\forall b\models q\left(a\mathrela Eb\right)$.\end{prop}
\begin{proof}
Since $E$ is bounded, $\equiv_{L,A}^{\alpha}$ refines it on $X$,
so if $a\equiv_{M}b$ for $a,b\in X$ then $a\mathrela Eb$.\end{proof}
\begin{rem}
\label{rem:G_delta}Suppose $Y\subseteq X$ is pseudo $G_{\delta}$.
For a model $M$, $Y_{M}$ is not necessarily $G_{\delta}$. But in
case $Y$ is $\equiv_{L,A}^{\alpha}$-invariant and $A\subseteq M$,
it is. Indeed, $\C^{\alpha}\backslash Y$ is pseudo $F_{\sigma}$,
and so $\left(\C^{\alpha}\backslash Y\right)_{M}$ is $F_{\sigma}$.
But since $\equiv_{M}$ refines $\equiv_{L,A}^{\alpha}$, $\left(\C^{\alpha}\backslash Y\right)_{M}\cap Y_{M}=\emptyset$.
In addition, if $A$, $T$ and $\alpha$ are countable, $Y$ is pseudo
closed and $\equiv_{L,A}^{\alpha}$-invariant, then $Y_{M}$ is $G_{\delta}$,
so $Y$ is pseudo $G_{\delta}$. In fact, in that case $Y$ is type
definable over $M$. 
\end{rem}
Assume that $E$ is pseudo $F_{\sigma}$. This is equivalent to saying
that there are $\emptyset$-type definable sets $U_{n}\subseteq\C^{\alpha\cdot2}$
for $n<\omega$ such that $E=\bigcup\set{U_{n}}{n<\omega}$ (this
follows by compactness, as $E$ is $\emptyset$-invariant). In this
case the set $U_{n}^{M}=\pi\left(U_{n,M}\right)\subseteq S_{\alpha}\left(M\right)^{2}$
is closed (where $\pi:S_{\alpha\cdot2}\left(M\right)\to S_{\alpha}\left(M\right)^{2}$
is the projection) and hence $E_{M}=\bigcup\set{U_{n}^{M}}{n<\omega}$
is $F_{\sigma}$. We assume that the sequence $\sequence{U_{n}}{n<\omega}$
is in \emph{normal form}, i.e., $U_{0}$ contains the diagonal $\Delta_{X}$,
$U_{n}$ is symmetric and: 
\[
U_{n}\circ U_{n}\upharpoonright X=\set{\left(a,b\right)\in X^{2}}{\exists c\in X\left(a,c\right)\in U_{n}\land\left(c,b\right)\in U_{n}}\subseteq U_{n+1}.
\]
So the $U_{n}$ are increasing on $X$. 
\begin{defn}
\label{def:strongly closed}Suppose $Y\subseteq X$ is $E$ invariant.
We say $E$ is \emph{strongly closed on $Y$} if there exists some
$n<\omega$ such that $E\upharpoonright Y=Y^{2}\cap U_{n}$. Note
that this may depend on the choice of the $U_{n}$'s. 
\end{defn}

\subsubsection{\label{sub:Countable-language}Countable language}

Suppose $T$ and $\alpha$ are countable. In this setting we will
translate our relation $E$ into an $F_{\sigma}$ relation on $X_{M}$,
as was done in \cite{PillayKrupinskiSolecki}.

For a countable model $A\subseteq M$, $S_{\alpha}\left(M\right)$
is Polish and if $Y$ is as in Remark \ref{rem:G_delta} then $Y_{M}$
is a Polish space (every $G_{\delta}$ set is), and similarly to \cite[Proposition 1.41]{BorelCard}
(with the same proof as there) we have:
\begin{prop}
\label{prop:G_delta changing the model-1} Fix a pseudo $G_{\delta}$
set $Y\subseteq X$, such that $Y$ is $E$-invariant. Then for any
two models $A\subseteq M,N$ we have:
\[
E^{M}\upharpoonright Y_{M}\sim_{B}E^{N}\upharpoonright Y_{N}.
\]

\end{prop}
So with this assumption and Proposition \ref{prop:G_delta changing the model-1},
we can refer to the Borel cardinality of the $F_{\sigma}$ equivalence
relation $E\upharpoonright Y$ without specifying the model.

\subsubsection{Countable or uncountable language}

Let $T$ be any complete first order theory and $\alpha$ any ordinal.
\begin{defn}
We say that a set $Y\subseteq\C^{\alpha}$ for some small $\alpha$
is \emph{pseudo strong Choquet} if $Y_{M}$ is strong Choquet for
all $M$.\end{defn}
\begin{example}
If $Y\subseteq\C^{\alpha}$ is pseudo closed or pseudo $G_{\delta}$
and $\equiv_{L}^{\alpha}$-invariant, then by Remark \ref{rem:G_delta}
and Proposition \ref{fac:G_delta subset is Choquet} it is pseudo
strong Choquet. \end{example}
\begin{rem}
For countable $T$ and $\alpha$, ``pseudo strong Choquet'' is the
correct analog of pseudo $G_{\delta}$ for $\equiv_{L}^{\alpha}$-invariant
sets. This follows from \cite[Theorem 8.17]{KechrisClassical}. 
\end{rem}

\subsubsection{\label{sub:First-application:-invariant}Invariant equivalence relations
with a nice automorphism group}

Let $C$ be some subset of $X$. Suppose that $\Gamma\leq\Aut\left(\C\right)$.
\begin{defn}
\label{def:proper}
\begin{enumerate}
\item A formula $\varphi\in L_{\alpha}\left(\C\right)$ is said to be\emph{
$C$-generic} if finitely many translates of $\varphi$ under the
action of $\Gamma$ cover $C$. 
\item The formula $\varphi$ is said to be\emph{ $C$-weakly generic} if
there is a non-$C$-generic formula $\psi\in L_{\alpha}\left(\C\right)$
such that $\varphi\vee\psi$ is $C$-generic.
\item A partial type $p\subseteq L_{\alpha}\left(\C\right)$ is said to
be $C$-generic ($C$-weakly generic) if all its formulas are. 
\item A partial type $p\subseteq L_{\alpha}\left(\C\right)$ which is is
closed under conjunctions is said to be \emph{$C$-proper} if there
is a non-$C$-generic formula $\psi$ such that for all $\varphi\in p$,
$\varphi\vee\psi$ is $C$-generic. In general, $p$ is $C$-proper
when its closure under finite conjunctions is.
\end{enumerate}
\end{defn}
For the most part we will omit $C$ from the notation.

For $n<\omega$, let $p_{n}\left(x,y\right)$ be the type defining
$U_{n}$. 
\begin{prop}
\label{prop:proper exists}Suppose that $\Gamma$ is $C$-transitive:
for all $a,b\in C$ there is some $\sigma\in\Gamma$ such that $\sigma\left(a\right)=b$.
Then, for some $n<\omega$ and for all $a\in C$, $p_{n}\left(x,a\right)$
is proper. Moreover, there is a formula $\psi\left(x,y\right)$ such
that $\psi\left(x,a\right)$ is the non-generic formula that witnesses
this.\end{prop}
\begin{proof}
First observe that if $p_{n}\left(x,a\right)$ is proper for some
$a\in C$, $\psi\left(x,a\right)$ witnesses this and $b\in C$, then
$p_{n}\left(x,b\right)$ is proper with $\psi\left(x,b\right)$ witnessing
it. So fix some $a\in C$.

Note that if $\psi\left(x,a\right)$ is not generic, then we can construct
inductively a sequence $a_{i}\in C$ for $i<\omega$ such that $\neg\psi\left(a_{i},a_{j}\right)$
for $j<i$: let $a_{0}=a$, and for $n+1$, let $\sigma_{0},\ldots,\sigma_{n}\in\Gamma$
be such that $\sigma_{i}\left(a\right)=a_{i}$ (so $\sigma_{0}=\id$)
and let $a_{n+1}\not\models\bigvee_{i\leq n}\sigma\left(\psi\left(x,a\right)\right)=\bigvee_{i\leq n}\psi\left(x,a_{i}\right)$.
By Ramsey and compactness (Fact \ref{fac:indiscernibles exist}),
there is an $A$-indiscernible sequence $\sequence{b_{i}}{i<\omega}$
in $X$ with the property that $\neg\psi\left(b_{i},b_{j}\right)$
for $j<i$. Here we used the fact that $X$ is type definable. 

Now suppose that for no $n<\omega$ is $p_{n}\left(x,a\right)$ proper.
This allows us to inductively construct formulas $\varphi_{n}\left(x,a\right)\in p_{n}\left(x,a\right)$
such that $\bigvee_{i<n}\varphi_{n}$ is not generic. By the remark
above and compactness, there is an $A$-indiscernible sequence $\sequence{b_{i}}{i<\omega}$
in $X$ such that for all $n<\omega$ and $j<i<\omega$, $\neg\varphi_{n}\left(b_{i},b_{j}\right)$.
But this means that $\left(b_{i},b_{j}\right)\notin U_{n}$ for all
$j<i<\omega$ and $n<\omega$, so $\neg E\left(b_{i},b_{j}\right)$.
By compactness, we may increasing the length of the sequence to any
length, contradicting the fact that $E$ is bounded on $X$. 
\end{proof}
Now assume that $C$ is $\Gamma$ invariant, and fix some $a\in C$.
By taking a countable union of models $M_{i}$ and a countable union
of subsets $\Gamma_{i}$ of $\Gamma$, we can find a model $M$ of
size $\left|A\right|+\left|L\right|+\left|\alpha\right|$ containing
$A$ and a subgroup $\Gamma^{*}\leq\Gamma$ of the size $\left|\alpha\right|+\left|L\right|$
such that:
\begin{enumerate}
\item $\left\{ a\right\} \cup A\subseteq M$. 
\item For all $\sigma\in\Gamma^{*}$, $\sigma\left(M\right)=M$ setwise.
\item If $\varphi$ is a formula over $a$ which is generic, then there
are finitely many elements from $\Gamma^{*}$ which witness this. 
\end{enumerate}
Recall that the Stone space $S_{\alpha}\left(M\right)$ has a natural
topology in which basic open sets are of the form $\left[\varphi\right]=\set{p\in S_{\alpha}\left(M\right)}{\varphi\in p}$.
When $r$ is a partial type, i.e., a consistent set of formulas over
$M$, we denote by $\left[r\right]$ the set $\set{p\in S_{\alpha}\left(M\right)}{r\subseteq p}$.
This set is compact. 

By (2) above, $\Gamma^{*}$ is a group of homeomorphisms of $S_{\alpha}\left(M\right)$. 
\begin{lem}
\label{lem:positive proper}Suppose $\left[a\right]_{E}\subseteq C\subseteq Y\subseteq X$
is $\Gamma$ invariant and that $\Gamma$ is $C$-transitive. Let
$\Uu$ be the family of open sets $U\subseteq Y_{M}$ for which there
is no finite set $\Delta\subseteq\Gamma^{*}$ with $C_{M}\subseteq\Delta\cdot U$
(all in the induced Stone space topology). Then for some $n<\omega$,
the compact set $\left[p_{n}\left(x,a\right)\right]\subseteq Y_{M}$
is $I_{\Uu}$-positive.\end{lem}
\begin{proof}
By Proposition \ref{prop:proper exists}, for some $n<\omega$, $p_{n}\left(x,a\right)$
is proper. By (3) above, if a formula $\varphi$ over $M$ is generic
then $\left[\varphi\right]\cap Y_{M}\notin\Uu$ and the converse also
holds. Unwinding the definitions, the proposition is clear. 
\end{proof}
Assume now that $E\upharpoonright C$ is not strongly closed, that
$\Gamma$ is $C$-transitive and that $C=\left[a\right]_{E}$. Since
$U_{n}$ is $\emptyset$-invariant and $\Gamma$ is $C$-transitive,
this means that for any $n<\omega$, there is some $b\in C$ such
that $\left(a,b\right)\notin U_{n}$. By enlarging $\Gamma^{*}$ and
$M$, we may assume:
\begin{enumerate}
\item [(4)]For all $n<\omega$ there is $\sigma\in\Gamma^{*}$ such that
$\left(a,\sigma\left(a\right)\right)\notin U_{n}$. 
\end{enumerate}
We are now ready to state our result:
\begin{thm}
\label{thm:countable, first case}Assume that $T$, $A\subseteq\C$
and $\alpha$ are countable. Suppose that:
\begin{enumerate}
\item $X\subseteq\C^{\alpha}$ is some type definable set over $A$.
\item $E$ is a pseudo $F_{\sigma}$ $\emptyset$-invariant equivalence
relation on $X$ with normal form $\sequence{U_{n}}{n<\omega}$ and
$E$ is bounded on $X$.
\item $C\subseteq X$ is an $E$ class, and $E\upharpoonright C$ is not
strongly closed (with respect to $\sequence{U_{n}}{n<\omega}$).
\item $C\subseteq Y\subseteq X$ is pseudo $G_{\delta}$ and $E$ invariant. 
\item $\Gamma\leq\Aut\left(\C\right)$ is $C$-transitive, and preserves
all $E$-classes (in particular, it preserves $X$).
\end{enumerate}
Then $E\upharpoonright Y$ is not smooth (see Proposition \ref{prop:G_delta changing the model-1}). \end{thm}
\begin{proof}
Keeping the notation from above, this follows directly from Theorem
\ref{thm:embedding2} with $X$ there being $Y_{M}$ (note that it
is $\Gamma^{*}$ invariant by assumptions (4) and (5) and that it
is Polish by (4) and Remark \ref{rem:G_delta}), $\Gamma$ there being
$\Gamma^{*}$ here, $Y$ there being $C_{M}$ here, $R_{n}$ there
being $U_{n}^{M}\upharpoonright Y_{M}$ here and $K$ there being
$\left[p_{k}\left(x,a\right)\right]$ for some $k<\omega$, chosen
by Proposition \ref{prop:proper exists} (note that as $Y$ contains
$C$, $Y_{M}$ contains $\left[p_{k}\left(x,a\right)\right]$, so
it is compact). By assumption (5), Theorem \ref{thm:embedding2}'s
$E_{\Gamma}^{X}$ is contained in $E_{M}\upharpoonright Y_{M}$, so
checking that the conditions of this theorem hold will suffice:

By Lemma \ref{lem:positive proper}, $K$ is $I_{\Uu}$-positive. 

Condition (a) there follows from assumption (3) here. Note that if
$p\in\left[p_{m}\left(x,a\right)\right]$, $q\mathrela{R_{n}}p$ and
$b\models q$ then $\left(a,b\right)\in U_{\max\left\{ n,m\right\} +1}$
(because there is some $b'\models q,c\models p$ such that $\left(b',c\right)\in U_{n}$,
but $\left(c,a\right)\in U_{m}$ so $\left(b',a\right)\in U_{\max\left\{ n,m\right\} +1}$
but $b\equiv_{a}b'$). So $q\in\left[p_{\max\left\{ n,m\right\} +1}\left(x,a\right)\right]$.
From this computation it follows that if $p\in\left[p_{k}\left(x,a\right)\right]$
and $q\mathrela{R_{n}^{\left(2\right)}}p$ for some $n\geq k$ then
for all $b\models q$, $\left(a,b\right)\in U_{n+2}$. So if $\sigma\in\Gamma^{*}$
is such that $\left(a,\sigma\left(a\right)\right)\notin U_{n+3}$
for $n\geq k$, then for all $p\in\left[p_{k}\left(x,a\right)\right]$,
$\left(p,\sigma\left(p\right)\right)\notin R_{n}^{\left(2\right)}$
(because for $b\models p$, $\left(\sigma\left(a\right),\sigma\left(b\right)\right)\in U_{k}$).

Condition (b) there follows similarly. As $\Gamma$ preserves $E$
classes, there is some $n<\omega$ such that $\left(a,\sigma\left(a\right)\right)\in U_{n}$.
So if $p\in\left[p_{k}\left(x,a\right)\right]$, then for all $b\models p$,
$\left(\sigma\left(b\right),b\right)\in U_{\max\left\{ n,k\right\} +3}$. 

Condition (c) there follows from the fact that $\Gamma\leq\Aut\left(\C\right)$
and that for all $n<\omega$, $U_{n}$ is $\emptyset$-invariant. \end{proof}
\begin{thm}
\label{thm:Choquet case, version 1}Let $T$, $A$ and $\alpha$ be
of any (small) size. Then under the same conditions as Theorem \ref{thm:countable, first case}
replacing (4) with: 
\begin{enumerate}
\item [(4)]$C\subseteq Y\subseteq X$ is pseudo strong Choquet and $E$
invariant.
\end{enumerate}
$E\upharpoonright Y$ has at least $2^{\aleph_{0}}$ classes. \end{thm}
\begin{proof}
Follows similarly from Theorem \ref{thm:embedding Choquet2} as in
the proof of Theorem \ref{thm:countable, first case}.\end{proof}
\begin{thm}
\label{thm:countable case, version 2} Suppose $T$, $A$ and $\alpha$
are countable, and the same assumptions as in Theorem \ref{thm:countable, first case}
hold, except (4) and (5) which we replace by:
\begin{enumerate}
\item [(4)] $C\subseteq Y\subseteq X$ is pseudo $G_{\delta}$ and $\Gamma$
invariant. 
\item [(5)]$\Gamma\leq\Aut\left(\C\right)$ is $C$-transitive, and for
all $\sigma\in\Gamma$ there is some $n<\omega$ such that for all
$c\in C$, $\left(c,\sigma\left(c\right)\right)\in U_{n}$. 
\end{enumerate}
Then $E\upharpoonright Y$ is not smooth. \end{thm}
\begin{proof}
To prove this theorem we could use either Theorem \ref{thm:embedding1}
or Theorem \ref{thm:embedding2} similarly to the proof of Theorem
\ref{thm:countable, first case}. The conditions there hold, but since
$\Gamma$ may not preserve $E$ classes, it is not clear that $E_{\Gamma}^{X}$
is contained in $E^{M}\upharpoonright Y_{M}$. To solve this problem,
we note that for any $x\in K$ (which is just $\left[p_{k}\left(x,a\right)\right]$
for some $a\in C$, $k<\omega$), $E_{\Gamma}^{X}\upharpoonright\closure{\Gamma\cdot x}$
is contained in $E^{M}\upharpoonright Y_{M}$, and recall that the
the image of the embedding $\phi$ of either Theorem \ref{thm:embedding1}
or Theorem \ref{thm:embedding2} is into $\closure{\Gamma\cdot x}$
for some $x\in K$. 

Indeed, fix some $p\in\left[p_{k}\left(x,a\right)\right]$ and $\sigma\in\Gamma^{*}$,
and let $n<\omega$ correspond to (5). Then for any $q\in\closure{\Gamma^{*}\cdot p}$,
$\left(\sigma\left(q\right),q\right)\in U_{n}^{M}$ as this is a closed
condition. 
\end{proof}
As above we give a general analog (using Theorem \ref{thm:embedding Choquet1}
or Theorem \ref{thm:embedding Choquet2}). Unfortunately, in this
case, being pseudo strong Choquet is not enough in order to prove
the theorem since we do not know that the range of $\phi$ can be
chosen to be $\closure{\Gamma\cdot x}$. 
\begin{thm}
\label{thm:Choquet case, version 2} Let $T$, $A$ and $\alpha$
be of any (small) size. Then under the same conditions as Theorem
\ref{thm:countable case, version 2} replacing (4) with: 
\begin{enumerate}
\item [(4)]$C\subseteq Y\subseteq X$ is pseudo closed and and $\Gamma$
invariant. 
\end{enumerate}
$E\upharpoonright Y$ has at least $2^{\aleph_{0}}$ classes. \end{thm}
\begin{cor}
For $E=\mathordi{\equiv_{L}^{\alpha}}$, the group $\Autf\left(\C\right)$
satisfies both the condition of Theorem \ref{thm:countable, first case}
and Theorem \ref{thm:countable case, version 2}, and so \cite[Main Theorems A and B]{BorelCard}
both follow directly. 

In addition, \cite[Fact 1.1]{BorelCard} has an obvious analog (at
least in the countable case) for the cases described in Theorem \ref{thm:countable, first case}
and Theorem \ref{thm:countable case, version 2}. In particular, in
these cases, an $E$ class is closed iff it is $G_{\delta}$ iff $E$
is strongly closed on it, and if $T$ is small and $\alpha$ is finite
then all classes are closed. 
\end{cor}
We can also deduce that \cite[Corollary 1.12]{Newelski} hold for
the cases described above (both for countable and uncountable languages),
which begs the question:
\begin{problem}
Do our result extend to any $\emptyset$-invariant $F_{\sigma}$ relation?\end{problem}
\begin{rem}
One of the properties of $\equiv_{L}$ is that if $a\equiv_{M}b$
for some model $M$, then $d_{L}\left(a,b\right)\leq2$ where $d_{L}$
is the Lascar metric. An analog for $E$ and its normal form would
be that for some $n<\omega$, if $M\supseteq A$ and $a\equiv_{M}b$
then $\left(a,b\right)\in U_{n}$. This has no reason to hold in general.
However, if $\Gamma$ is $C$-transitive then for some $n<\omega$
and all $M$ and $\Gamma^{*}$ as in (1)--(3) above, there is a nonempty
$\Gamma^{*}$-invariant closed subset $S\subseteq S_{X}\left(M\right)$
such that for any $p\in S\cap C_{M}$, if $b,c\models p$ then $\left(b,c\right)\in U_{n+1}$.
Moreover, it is dense in the following sense: for every $b\in C$,
there is some $c\in C$ such that $\tp\left(c/M\right)\in S$ and
$\left(b,c\right)\in U_{n}$.

Indeed, let $n<\omega$ be such that $p_{n}\left(x,a\right)$ is proper
for all $a\in C$, and let $\psi\left(x,y\right)$ be the formula
that witnesses this (see Proposition \ref{prop:proper exists}). Let
$a\in C$, $M$ and $\Gamma^{*}$ be as in (1)--(3). Let $S$ be the
set of types $\left[\set{\neg\psi\left(x,\sigma\left(a\right)\right)}{\sigma\in\Gamma^{*}}\right]$.
This is obviously closed and $\Gamma^{*}$-invariant. Suppose $p\in S\cap C_{M}$
and $b,c\models p$. We will show that $\left(b,c\right)\in U_{n+1}$,
i.e., $\left(b,c\right)\models p_{n+1}$. Let $\xi\left(x,y\right)\in p_{n+1}$,
and let $\chi\left(x,y\right)\in p_{n}$ be such that $\chi\left(x,y\right)\wedge\chi\left(z,y\right)\to\xi\left(x,z\right)$.
Since $\chi\left(x,a\right)\vee\psi\left(x,a\right)$ is generic,
for some $\sigma\in\Gamma^{*}$, $b,c\models\chi\left(x,\sigma\left(a\right)\right)\vee\psi\left(x,\sigma\left(a\right)\right)$,
but by the definition of $S$, $b,c\models\chi\left(x,\sigma\left(a\right)\right)$.
It follows that $\xi\left(b,c\right)$ holds. 

We also need to show the denseness property. Fix some $b\in C$. It
is enough to show that the set $\set{\neg\psi\left(x,\sigma\left(a\right)\right)}{\sigma\in\Gamma^{*}}\cup p_{n}\left(x,b\right)$
is consistent. Suppose not, so for some $\xi\left(x,y\right)\in p_{n}\left(x,y\right)$
and some finite $\Delta\subseteq\Gamma^{*}$, $\xi\left(x,b\right)\to\bigvee_{\sigma\in\Delta}\psi\left(x,\sigma\left(a\right)\right)$.
Since $\Gamma$ is $C$-transitive, for some $\tau\in\Gamma$, $\tau\left(b\right)=a$,
so $\xi\left(x,a\right)$ implies $\bigvee_{\sigma\in\Delta}\psi\left(x,\tau\circ\sigma\left(a\right)\right)$.
But then $\bigvee_{\sigma\in\Delta}\psi\left(x,\tau\circ\sigma\left(a\right)\right)\vee\psi\left(x,a\right)$
is generic --- contradiction. 

This observation could have been used in the proof of e.g., Theorem
\ref{thm:countable, first case}, using $S\cap Y_{M}$ as our Polish
space. 
\end{rem}

\subsubsection{Definable and type definable group action }
\begin{defn}
For an ordinal $\beta$, $\left(H,\cdot\right)$ is a \emph{type definable
group} contained in $\C^{\beta}$ when $H$ is type definable and
the multiplication is type definable. 
\end{defn}
Suppose $\left(H,\cdot\right)$ is a type definable group over $\emptyset$.
Let $G$ be an $\emptyset$-invariant pseudo $F_{\sigma}$ subgroup.
In this case $G$ has a normal form: $G=\bigcup\set{V_{n}}{m<\omega}$
where $V_{n}$ is $\emptyset$-type definable, $\left\{ e\right\} \in V_{n}$,
$V_{n}$ is symmetric ($V_{n}=V_{n}^{-1}$), and $V_{n}^{\cdot2}\subseteq V_{n+1}$.

Suppose that $X\subseteq\C^{\alpha}$ is $\emptyset$-type definable
and that $*$ is an $\emptyset$-type definable group action of $H$
on $X$. In particular, the orbit equivalence relation of the action
$E_{H}^{X}$ is a closed invariant equivalence relation on $X$ and
$E_{G}^{X}$ is an $\emptyset$-invariant pseudo $F_{\sigma}$ equivalence
relation in the sense discussed in the previous subsection, with normal
form defined by: 
\[
U_{n}=\set{\left(a,b\right)\in X\x X}{\exists g\in V_{n}\left(g*a=b\right)}
\]
 for $n<\omega$. 
\begin{defn}
To simplify notation, we call such a tuple $\bar{D}=\left(\alpha,\beta,G,H,\sequence{V_{n},U_{n}}{n<\omega},\cdot,X,*\right)$
an\emph{ $F_{\sigma}$ action}. If $E_{G}^{X}$ is bounded, we call
$\bar{D}$ a \emph{bounded $F_{\sigma}$ action}. \end{defn}
\begin{example}
For an $\emptyset$-type definable group $G\subseteq\C^{\alpha}$,
$G_{\emptyset}^{000}$ is defined as the smallest bounded index invariant
subgroup of $G$ and it is generated by the set $\set{a^{-1}\cdot b}{a\equiv_{L}^{\alpha}b,\, a,b\in G}$.
So, letting $W_{n}=\set{a^{-1}\cdot b}{d_{L}\left(a,b\right)\leq n}$
where $d_{L}$ is the Lascar distance, we see that $G=\bigcup_{n<\omega}V_{n}$
where $V_{n}=\set{\prod_{i<n}c_{i}^{\pm1}}{c_{i}\in W_{n}}$. See
\cite{gismatulin} for more. So $\left(\alpha,\alpha,G_{\emptyset}^{000},G,\sequence{V_{n},U_{n}}{n<\omega},\cdot,G,\cdot\right)$
is a bounded $F_{\sigma}$ action. 
\end{example}
We shall need a technical assumption that seems necessary for this
approach to work. 
\begin{defn}
\label{def:normal}We say that $a\in X$ is \emph{shifty} if one of
the following holds:
\begin{enumerate}
\item (\emph{Right} shifty) For every $k<\omega$ there exists $n=n_{k}<\omega$
such that for any $g_{1},g_{2}\in H$ if $\left(g_{1}*a,g_{2}*a\right)\in U_{k}$
then $\left(\left(g_{1}\cdot g_{2}^{-1}\right)*a,a\right)\in U_{n}$
\uline{or}:
\item (\emph{Left} shifty) For every $k<\omega$ there exists $n=n_{k}<\omega$
such that for any $g_{1},g_{2}\in H$ if $\left(g_{1}*a,g_{2}*a\right)\in U_{k}$
then $\left(\left(g_{1}^{-1}\cdot g_{2}\right)*a,a\right)\in U_{n}$
and if $\left(\left(g_{1}^{-1}\cdot g_{2}\right)*a,a\right)\in U_{k}$
then $\left(g_{1}*a,g_{2}*a\right)\in U_{n}$.
\end{enumerate}
\end{defn}
\begin{rem}
In both cases, we may safely assume that $n_{k}\geq k$. \end{rem}
\begin{example}
Suppose $a\in X$ and $\stab_{H}\left(a\right)\trianglelefteq H$.
Then $a$ is right shifty. \end{example}
\begin{proof}
Let $k<\omega$ be given and let $n=k$. If $g_{2}*a=\left(h\cdot g_{1}\right)*a$
for $h\in V_{n}$ then $\left(g_{2}^{-1}\cdot h\cdot g_{1}\right)*a=a$
and since $\stab_{H}\left(a\right)$ is normal, $\left(h\cdot g_{1}\cdot g_{2}^{-1}\right)*a=a$
so $\left(g_{1}\cdot g_{2}^{-1}\right)*a=h^{-1}*a$. As $V_{n}$ is
symmetric, we are done. \end{proof}
\begin{example}
Suppose that for every $k<\omega$ there exists $n<\omega$ such that
for any $c,d\in G*a$ and $g\in H$, if $\left(c,d\right)\in U_{k}$
then $\left(g*c,g*d\right)\in U_{n}$. Then $a$ is left shifty. This
happens for instance when $V_{n}$ is definable for all $n<\omega$
and $G$ is a normal subgroup of $H$.\end{example}
\begin{proof}
If $V_{n}$ is definable, then by compactness for every $k<\omega$
there is some $n<\omega$ such that for all $g\in H$, $gV_{k}g^{-1}\subseteq V_{n}$.
So if there is some $h\in V_{k}$ such that $c=h*d$, then $g*c=\left(g\cdot h\right)*d$,
but $g\cdot h=h'\cdot g$ for $h'\in V_{n}$ so $\left(g*c,g*d\right)\in U_{n}$.\end{proof}
\begin{lem}
\label{lem:formula multiplication-1}Suppose $\varphi\left(x,y\right)$
is some formula where $x$ comes from the first $\alpha$ variables.
Then there is a formula $\psi\left(x',y,z\right)$ with $x'$ coming
from the first $\alpha$ variables such that for every $g\in H$,
and any $a\in\C^{\lg\left(y\right)}$, $g*\left(\varphi\left(\C^{\alpha},a\right)\cap X\right)=\left(\psi\left(\C^{\alpha},a,g\right)\cap X\right)$.\end{lem}
\begin{proof}
If $\alpha$, and $\beta$ were finite, so that $*$ and $\cdot$
were definable, then we could just define $\psi\left(x,y,z\right)=\varphi\left(z^{-1}\cdot x,y\right)$
(so $x=x'$). Otherwise, it is a standard compactness argument. Note
that we need that both $X$ and $H$ are closed. 
\end{proof}
Lemma \ref{lem:formula multiplication-1} defines an action of $H$
on sets of the form $X\cap\varphi\left(\C^{\alpha}\right)$. In order
to ease notation, we will write $g*\varphi$ instead of $g*\left(\varphi\left(\C^{\alpha}\right)\cap X\right)$.
This induces a natural action of $H$ on the set of types in $X$.
If $\dcl\left(A\right)=A$, then $H\cap A$ (and also $G\cap A$)
is a subgroup of $H$, and so it acts naturally by homeomorphisms
on $S_{X}\left(A\right)$ (with the usual Stone topology). In that
case, for any $g\in H\cap A$, $c\models p$ iff $g*c\models g*p$. 

Fix an $E_{G}^{X}$-class $C\subseteq X$. Similarly to Definition
\ref{def:proper}, we define $C$-generic and $C$-weakly generic
formulas and $C$-proper types, replacing the action of an automorphism
group $\Gamma$ by the action of $G$ on $L_{\alpha}\left(\C\right)$
(note: $G$ and not $H$). We omit the details, since it is exactly
as above. 

For $n<\omega$, let $p_{n}\subseteq L_{\beta}\left(\emptyset\right)$
be the partial types defining $V_{n}$ and let $q_{n}\left(x,a\right)$
be the partial type saying $x\in X\land\exists g\in V_{n}\left(g*a=x\right)$.
\begin{lem}
\label{lem:proper types exist group}Suppose $a\in X$ is shifty.
Then, for some $n<\omega$, $q_{n}\left(x,a\right)$ is a $G*a$-proper
type. \end{lem}
\begin{proof}
The proof uses the same basic idea as in Lemma \ref{prop:proper exists},
but one has to be a bit careful.

Assume first that $a$ is right shifty. Suppose $\pi_{*}$ is the
partial type defining $*$ and that $\pi_{X}$ is the type defining
$X$. We may assume that these types, as well as $p_{n}$ and $q_{n}$
are closed under conjunctions. First we need to establish the following:
\begin{claim*}
For each $k<\omega$ there is some $n<\omega$ such that for all formulas
$\varphi\in p_{n}$, $\theta\in\pi_{*}$ there are formulas $\psi\in p_{k}$
and $\theta'\in\pi_{*}$ such that for every $g_{1},g_{2}\in H$,
if
\[
\exists z\left(\psi\left(z\right)\land\theta'\left(z,g_{1}*a,g_{2}*a\right)\right)
\]
then 
\[
\exists z\left(\varphi\left(z\right)\land\theta\left(z,a,\left(g_{1}\cdot g_{2}^{-1}\right)*a\right)\right).
\]
\end{claim*}
\begin{proof}
[Proof of claim.] Let $k<\omega$ be given, and let $n<\omega$ be
the corresponding number from Definition \ref{def:normal}. Then the
following is inconsistent: there are $g_{1},g_{2}\in H$ such that
$g_{2}*a\in V_{k}*\left(g_{1}*a\right)$ but $\left(g_{1}\cdot g_{2}^{-1}\right)*a\notin V_{n}*a$.
Applying compactness, we are done. Note that these formulas may depend
on $a$ (but not on $g_{1},g_{2}$). 
\end{proof}
Assume that for all $n<\omega$, $q_{n}$ is not proper. For each
$k<\omega$, let $n_{k}<\omega$ be the corresponding number from
the claim. 

Since $q_{n_{k}}$ is not proper for all $k<\omega$, we can find
formulas $\varphi_{k}\in p_{n_{k}}$ and $\theta_{k}\in\pi_{*}$ such
that $\bigvee_{k<m}\psi_{k}'$ is not generic for all $m<\omega$
where $\psi_{k}'\left(x\right)=\exists y\left(\varphi_{k}\left(y\right)\land\theta_{k}\left(y,a,x\right)\right)$.
(note: a formula in $q_{n_{k}}$ generally looks like $\psi'_{k}\wedge\tau$
for $\tau\in\pi_{X}$, but this does not matter for genericity.) 

For each $k<\omega$, the claim provides formulas $\psi_{k}\in p_{k}$
and $\theta'_{k}\in\pi_{*}$ such that:

If $g_{1},g_{2}\in H$ and $g_{2}^{-1}*a\notin g_{1}^{-1}*\psi_{k}'$
then $\neg\exists z\left(\psi_{k}\left(z\right)\land\theta_{k}'\left(z,g_{1}*a,g_{2}*a\right)\right)$.
Note that this latter condition implies that $\left(g_{1}*a,g_{2}*a\right)\notin U_{k}$. 

Fix some $n<\omega$ and let $\psi'=\bigvee_{k<n}\psi_{k}'$. Since
$\psi'$ is not generic, there is a sequence $\sequence{g_{i}\in G}{i<\omega}$
such that $g_{i}^{-1}*a\notin g_{j}^{-1}*\psi'$ for $j<i$. This
means that $\left(g_{j}*a,g_{i}*a\right)\notin U_{k}$ for $k<n$,
and for each $k$, this is because of $\psi_{k}$ and $\theta_{k}'$.
Note that although $V_{k}\subseteq V_{k+1}$, we do not get that $\psi_{k}$
implies $\psi_{k+1}$, so we really need to keep all the formulas. 

Now, by compactness we can find a sequence $\sequence{a_{i}\in X}{i<\omega}$
such that for all $j<i<\omega$, $\left(a_{j},a_{i}\right)\notin U_{k}$
for all $k<\omega$ (and each time because of the same formulas).
By Ramsey and compactness (Fact \ref{fac:indiscernibles exist}) we
may assume that this sequence is indiscernible. But this is a contradiction
to our assumption that the action is bounded. 

If $a$ is left shifty, the proof is exactly the same, replacing $g_{i}^{-1}$
by $g_{i}$.
\end{proof}
Recall that if $E_{G}^{X}\upharpoonright G*a$ is not strongly closed
for some $a\in X$ (see Definition \ref{def:strongly closed}), then
for all $n<\omega$ there are $g_{1},g_{2}\in G$ such that $\left(g_{1}*a,g_{2}*a\right)\notin U_{n+1}$.
But then either $\left(a,g_{1}*a\right)\notin U_{n}$ or $\left(a,g_{2}*a\right)\notin U_{n}$.
So we may always assume that $g_{1}=e$. 
\begin{thm}
\label{thm: countable groups} Suppose $T$ is a complete countable
first-order theory, $\alpha,\beta$ countable ordinals. Suppose that
$\left(\alpha,\beta,G,H,\sequence{V_{n},U_{n}}{n<\omega},\cdot,X,*\right)$
is a bounded $F_{\sigma}$ action and suppose $Y\subseteq\C^{\alpha}$
is a pseudo $G_{\delta}$ set contained in $X$ which is $E_{G}^{X}$
invariant. If for some \uline{shifty} $a\in Y$, $E_{G}^{X}\upharpoonright G*a$
is not strongly closed, then $E_{G}^{X}\upharpoonright Y$ is non-smooth. \end{thm}
\begin{proof}
This follows from Theorem \ref{thm:embedding1}, just like the proof
of Theorem \ref{thm:countable, first case}. By Lemma \ref{lem:proper types exist group},
for some $k<\omega$, $q_{k}\left(x,a\right)$ is $G*a$-proper, and
this is witnessed by some non-generic formula $\psi$. Construct recursively
a countable model $M$ such that:
\begin{enumerate}
\item $a\in M$ and $\psi$ is over $M$.
\item If $\varphi\in L_{\alpha}\left(M\right)$ is $G*a$-generic, then
for some $\Delta\subseteq G\cap M$, $\Delta*\varphi$ contains $G*a$.
\item For all $n<\omega$, there is some $g\in G\cap M$ such that $\left(a,g*a\right)\notin U_{n}$. 
\end{enumerate}
In the language of Theorem \ref{thm:embedding1}, $X$ is $Y_{M}$,
$\Gamma$ is $G\cap M$, $Y$ is $\left(G*a\right)_{M}$, $R_{n}$
is $U_{n}^{M}\cap Y_{M}^{2}$ and $K$ is the compact set $\left[q_{k}\left(x,a\right)\right]$.
The fact that $\left[q_{k}\left(x,a\right)\right]$ is $I_{\Uu}$-positive
follows from (2) and the fact that $q_{k}\left(x,a\right)$ is proper
(see the proof of Lemma \ref{lem:positive proper}).

Condition (a) follows from (3) above: if $\left(a,g*a\right)\notin U_{N}$
for $N$ big enough, then for all $p\in Y_{M}$ containing $q_{k}\left(x,a\right)$,
$\left(g*p,p\right)\notin U_{n}^{M,\left(4\right)}$. We illustrate
this: if $\left(p,g*p\right)\in U_{n}^{M,\left(2\right)}$, then for
some $q\in X_{M}$,$\left(p,q\right)\in U_{n}^{M}$ and $\left(q,g*p\right)\in U_{n}^{M}$.
So for some $b_{1},b_{2}\models p$ and $c_{1},c_{2}\models q$, $\left(b_{1},c_{1}\right)\in U_{n}$,
$\left(c_{2},g*b_{2}\right)\in U_{n}$. Since $\left(b_{1},a\right)\in U_{k}$,
and $b_{1}\equiv_{a}b_{2}$, $\left(b_{1},b_{2}\right)\in U_{k+1}$.
Similarly, $\left(c_{1},c_{2}\right)\in U_{\max\left\{ n,k\right\} +2}$.
It follows that $\left(b_{2},g*b_{2}\right)\in U_{\max\left\{ k,n\right\} +4}$.
Suppose $a$ is right shifty. As $b_{2}\in G*a$ and since $a$ is
right shifty, we get that $\left(a,g*a\right)\in U_{N}$ for some
$N$. If $a$ is left shifty, then as $\left(a,b_{2}\right)\in U_{k}$,
$\left(g*a,g*b_{2}\right)\in U_{n_{k}}$ for some large $n_{k}$,
so $\left(a,g*a\right)\in U_{n_{k}+2}$.

Condition (b) is trivial, since any $g\in G\cap M$ belongs to some
$V_{n}$. \end{proof}
\begin{problem}
Is shiftiness of $a$ necessary? \end{problem}
\begin{cor}
\label{cor:closed or non-smooth}With the same assumptions of Theorem
\ref{thm: countable groups}, if the action of $G$ is free (if $g*x=h*x$
then $h=g$) then either $G=V_{n}$ for some $n<\omega$, in which
case $E_{G}^{X}$ is strongly closed so smooth or $E_{G}^{X}\upharpoonright Y$
is non-smooth. \end{cor}
\begin{proof}
Note that by assumption, every $a\in X$ is right shifty (since $\stab_{H}\left(a\right)=e$
is a normal subgroup). Now, if $E_{G}^{X}\upharpoonright Y$ is smooth,
then by Theorem \ref{thm: countable groups}, for every $a\in Y$,
$E\upharpoonright G*a$ is strongly closed. So for some $a\in Y$
and $n<\omega$ for all $b\in Y$, $a\mathrela{E_{G}^{X}}b$ iff $\left(a,b\right)\in U_{n}$.
Since the action is free, it follows that then $G=V_{n}$. \end{proof}
\begin{thm}
\label{Thm:groups - general language} Suppose $T$ is a complete
first-order theory, $\alpha,\beta$ small ordinals. Suppose that $ $$\left(\alpha,\beta,G,H,\sequence{V_{n},U_{n}}{n<\omega},\cdot,X,*\right)$
is a bounded $F_{\sigma}$ action and suppose $Y\subseteq\C^{\alpha}$
is a pseudo a strong Choquet set contained in $X$ which is $E_{G}^{X}$
invariant. Suppose also that for some \uline{shifty} $a\in Y$,
$E_{G}^{X}\upharpoonright G*a$ is not strongly closed. Then $\left|Y/E_{G}^{X}\right|\geq2^{\aleph_{0}}$. \end{thm}
\begin{proof}
Follows similarly from Theorem \ref{thm:embedding Choquet1}. 
\end{proof}
We can also recover Newelski's results \cite[Theorem 3.1]{Newelski}
about groups generated by countably many type definable sets over
$\emptyset$.
\begin{cor}
\label{cor:Newelski}Let $T$ be any first order theory and $\alpha$
any small ordinal. Suppose that $\left(H,\cdot\right)$ is a $\emptyset$-type
definable group such that $H\subseteq\C^{\alpha}$. Suppose that $G\leq H$
is a subgroup which is generated by countably many sets $V_{n}$ for
$n<\omega$ which are $\emptyset$-type definable. Suppose that $G\leq H_{0}\leq H$
is a subgroup which is pseudo $G_{\delta}$ or pseudo closed (or even
pseudo strong Choquet). Assume also that $\left[H:G\right]$ is bounded.
Then:
\begin{enumerate}
\item If $G$ is pseudo $G_{\delta}$ or pseudo closed or even pseudo strong
Choquet then $G$ is pseudo closed and in fact generated by finitely
many of the sets $V_{n}$ in finitely many steps.
\item If $G$ is not pseudo closed then $\left[H_{0}:G\right]\geq2^{\aleph_{0}}$. 
\item If $T$ and $\alpha$ are countable then either $G$ is pseudo closed
or the equivalence relation $E_{G}^{H_{0}}$ on $H_{0}$ of being
in the same coset modulo $G$ is not smooth. 
\item If $T$ is small and $\alpha$ is finite then $G$ is pseudo closed. 
\item If we remove the assumption that $\left[H:G\right]$ is bounded, we
still get (1) for pseudo closed $G$.
\end{enumerate}
\end{cor}
\begin{proof}
We may assume that $V_{n}$ is symmetric ($V_{n}=V_{n}^{-1}$), $V_{0}=\left\{ 1_{H}\right\} $,
and $V_{n}^{\cdot2}\subseteq V_{n+1}$. Consider the action of $G$
on $H$ by left multiplication and the orbit equivalence relation
$E_{G}^{H}$ on $H$. Then by Theorem \ref{Thm:groups - general language}
with $X=H,Y=G$, $\alpha=\beta$, $a=e_{G}$ we get (1). Applying
it again with $X=H,Y=H_{0}$ we get (2). (3) follows from Corollary
\ref{cor:closed or non-smooth}. 

(4) Suppose not. Since $T$ is small, the set $S_{\alpha}\left(\emptyset\right)$
is countable. Thus every subset of it is $G_{\delta}$, in particular
the set 
\[
Q=\set{q\in S_{\alpha}\left(\emptyset\right)}{\forall b\models q\,\left(b\in G\right)}.
\]

But then $G$ is pseudo $G_{\delta}$ so it is pseudo closed by (1). 

(5) Note that in that case $G$ is type definable over $\emptyset$,
so we can replace $H$ by $G$. \end{proof}
\begin{cor}
\label{cor:invariant + finite index implies definable}Suppose $H$
is a definable group, and $G$ an $\emptyset$-invariant subgroup
which is a union of countably many type definable sets. Then if $\left[H:G\right]<\infty$
then $G$ is definable. \end{cor}
\begin{proof}
By Corollary \ref{cor:Newelski} (2), $G$ must be type definable
over $\emptyset$. But then its complement is also type definable
since it is a finite union of type definable sets, so it is definable
by compactness. \end{proof}
\begin{example}
(with Pierre Simon) Theorem \ref{thm: countable groups} does not
hold in a very strong sense, if the group $G$ is only $\emptyset$-invariant
and not pseudo $F_{\sigma}$. More precisely there is a countable
theory $T$ where Corollary \ref{cor:invariant + finite index implies definable}
fails.

Let $T$ be the theory of an infinite dimensional vector space over
$\Ff_{2}$ in the language $\left\{ +,0\right\} $. Add predicates
$U_{n}$ to the language and add axioms saying that $U_{n}$ are independent
subspaces of co-dimension $1$ (independent in the sense that any
finite Boolean combination is nonempty). Then $T$ is consistent as
one can take for $U_{n}$ the kernels of independent functionals.
Let $\C$ be a monster model for $T$, and let $H$ be the group $\left(\C,+\right)$.
Let $G$ be the intersection $\bigcap\set{U_{n}}{n<\omega}$. Then
the index $\left[H:G\right]=2^{\aleph_{0}}$. In fact, the cosets
of $G$ in $H$ are exactly the types $X_{\eta}=\bigcap\set{U_{n}^{\eta\left(n\right)}}{n<\omega}$
where $\eta:\omega\to2$ and $U_{n}^{0}=U_{n}$, $U_{n}^{1}=\C\backslash U_{n}$.
Pick a basis $\set{v_{i}}{i<2^{\aleph_{0}}}$ for the space $H/G$.
Any map $\eta:2^{\aleph_{0}}\to2$ defines a subspace $V_{\eta}$
by taking the kernel of the functional mapping $v_{i}$ to $\eta\left(i\right)$.
Obviously, if $\eta$ is not trivial then $\left[H:\pi^{-1}\left(V_{\eta}\right)\right]=2$
(where $\pi$ is the projection $H\to H/G$). So for at least one
$\eta$, $\pi^{-1}\left(V_{\eta}\right)$ is not definable. But all
of them are invariant as they are union of cosets $X_{\eta}$. 
\end{example}

\begin{example}
Corollary \ref{cor:Newelski} (5) does not hold when $G$ is pseudo
$G_{\delta}$. For instance, let $T=RCF$, and add to the language
constant symbols for the rational numbers $\mathbb{Q}$. Then $\mathbb{Q}$
itself is pseudo open (in every model), so also pseudo $G_{\delta}$,
but definitely not closed (every closed infinite subset must have
unbounded cardinality). 
\end{example}
\bibliographystyle{alpha}
\bibliography{common}

\begin{thebibliography}{KMS13}

\bibitem[Gis11]{gismatulin}
Jakub Gismatullin.
\newblock Model theoretic connected components of groups.
\newblock {\em Israel J. Math.}, 184:251--274, 2011.

\bibitem[HKL90]{harringtonKechrisLouveau}
L.~A. Harrington, A.~S. Kechris, and A.~Louveau.
\newblock A {G}limm-{E}ffros dichotomy for {B}orel equivalence relations.
\newblock {\em J. Amer. Math. Soc.}, 3(4):903--928, 1990.

\bibitem[Kec95]{KechrisClassical}
Alexander~S. Kechris.
\newblock {\em Classical descriptive set theory}, volume 156 of {\em Graduate
  Texts in Mathematics}.
\newblock Springer-Verlag, New York, 1995.

\bibitem[Ker07]{AvivThesis}
Aviv Keren.
\newblock Equivalence relations \& topological automorphism groups in simple
  theories.
\newblock Master's thesis, Hebrew University of Jerusalm, Israel, 2007.

\bibitem[KMS13]{BorelCard}
Itay Kaplan, Benjamin Miller, and Pierre Simon.
\newblock The borel cardinality of lascar strong types.
\newblock 2013.
\newblock \url{arXiv:1301.1197}.

\bibitem[KPS12]{PillayKrupinskiSolecki}
Krzysztof Krupinski, Anand Pillay, and Slawomir Solecki.
\newblock Borel equivalence relations and lascar strong types.
\newblock 2012.
\newblock \url{arXiv:1204.3485}.

\bibitem[New03]{Newelski}
Ludomir Newelski.
\newblock The diameter of a {L}ascar strong type.
\newblock {\em Fund. Math.}, 176(2):157--170, 2003.

\bibitem[Sil80]{silver}
Jack~H. Silver.
\newblock Counting the number of equivalence classes of {B}orel and coanalytic
  equivalence relations.
\newblock {\em Ann. Math. Logic}, 18(1):1--28, 1980.

\bibitem[TZ12]{TentZiegler}
Katrin Tent and Martin Ziegler.
\newblock {\em A Course in Model Theory (Lecture Notes in Logic)}.
\newblock Cambridge University Press, 2012.

\end{thebibliography}

\end{document}